\documentclass[12pt]{article}        

 \usepackage[T1]{fontenc}             
 \usepackage[width=16cm, height=22cm]{geometry}   

 \usepackage{mathtools} 
 \usepackage{amssymb}          
\usepackage{hyperref} 
 \usepackage[amsmath,thmmarks,hyperref]{ntheorem}

 \usepackage{icomma}                   
 \usepackage{enumerate}          
 \usepackage{color}                 
 \usepackage{setspace}                 
 \usepackage{needspace}               
 \usepackage{url}                    
 \usepackage{mathrsfs}                
 \usepackage{esint}

 \numberwithin{equation}{section}
 
 \usepackage[numbers,square]{natbib}
\theoremstyle{nonumberplain}  
\theoremheaderfont{\itshape}  
\theorembodyfont{\normalfont}  
\theoremseparator{.}  
\theoremsymbol{$\Box$}
\newtheorem{proof}{Proof} 
  
\theoremstyle{plain}  
\theoremheaderfont{\normalfont\bfseries}  
\theorembodyfont{\itshape}  
\theoremsymbol{~}
\theoremseparator{.}  
\newtheorem{proposition}{Proposition}[section]  
\newtheorem{corollary}[proposition]{Corollary}  
\newtheorem{lemma}[proposition]{Lemma}

\theorembodyfont{\normalfont}  
 
\newtheorem{remark}[proposition]{Remark}

\theoremstyle{nonumberplain}
\theoremheaderfont{\normalfont\bfseries}  
\theorembodyfont{\itshape}  
\theoremsymbol{~}
\theoremseparator{.}  
\newtheorem{theoremA}[proposition]{Theorem A}  
\newtheorem{theoremB}[proposition]{Theorem B} 
\newtheorem{theoremC}[proposition]{Theorem C} 
\newtheorem{theoremnothing}[proposition]{Theorem} 

\newcommand{\R}{\mathbb{R}}

\newcommand{\dd}{\mathrm{d}}

\newcommand{\<}{\left\langle}
\renewcommand{\>}{\right\rangle}

\newcommand{\m}{\mathfrak{m}}

\title{The Trace and the Mass of subcritical GJMS Operators}
\author{ Matthias Ludewig}

\begin{document}

\maketitle
 
\begin{center}
  Max-Planck Institut f\"ur Mathematik \\ 
  Vivatsgasse 7 / 53111 Bonn, Germany \\ \medskip
 maludewi@mpim-bonn.mpg.de \\ \medskip
\end{center}

\begin{abstract}
 \noindent Let $L_g$ be the subcritical GJMS operator on an even-dimensional compact \mbox{mani}fold $(X, g)$ and consider the zeta-regularized trace $\mathrm{Tr}_\zeta(L_g^{-1})$ of its inverse. We show that if $\ker L_g = 0$, then the supremum of this quantity, taken over all metrics $g$ of fixed volume in the conformal class, is always greater than or equal to the corresponding quantity on the standard sphere. Moreover, we show that in the case that it is strictly larger, the supremum is attained by a metric of constant mass. Using positive mass theorems, we give some geometric conditions for this to happen.
\end{abstract}

\section{Introduction}

On any compact Riemannian manifold $(X, g)$, there exists a sequence of natural conformally covariant differential operators $L^{(m)}_g = \Delta_g^m + \text{lower order}$, named GJMS operators after Graham, Jenne, Mason and Sparling, who first constructed them \cite{GJMS1}. Here the order $2m$ can be arbitrary if the dimension $n$ of $X$ is odd, but if $n$ is even, one has the restriction $1 \leq m \leq \frac{n}{2}$ in general. In particular, for $m=1$, we have $L^{(1)}_g = Y_g$, the Yamabe operator, which is famously connected to the problem of finding a conformal metric on $X$ with constant scalar curvature. The operator $L^{(n/2)}_g$ of order $2m=n$ is often referred to as the {\em critical} GJMS operator; similarly, we will usually refer to the operator of order $2m=n-2$ as the {\em subcritical} GJMS operator. This subcritical case will be our main object of study.  Notice that the Yamabe operator is  subcritical in dimension four, while the so-called {\em Paneitz operator} $L^{(2)}_g$ is subcritical in dimension six.

The main purpose of this paper is to prove the following result regarding the zeta-regularized trace of the inverse of the subcritical GJMS operator in even dimensions which says that the round sphere minimizes the quantity $\sup \mathrm{Tr}_\zeta(L_g^{-1})$ among all Riemannian manifolds, when the supremum is taken over all conformal metrics of volume equal to the volume of the sphere.

\begin{theoremA}
Let $(X, g_0)$ be a Riemannian manifold of even dimension $n\geq 4$ and let $L_{g_0} = L_{g_0}^{(n/2-1)}$ be the subcritical GJMS operator. Suppose that $\ker L_{g_0} = 0$. Then we have
\begin{equation} \label{TraceInequality1}
 \sup \mathrm{Tr}_\zeta(L_g^{-1}) \geq \mathrm{Tr}_\zeta(L_{g_{\mathrm{std}}}^{-1}),
\end{equation}
where $L_{g_{\mathrm{std}}}$ denotes subcritical GJMS operator on the standard sphere $(S^n, g_{\mathrm{std}})$ and the supremum is taken over all metrics $g$ in the conformal class of $g_0$ with volume equal to $\omega_n$, the volume of the standard sphere. Moreover, if $X$ is connected and the inequality \eqref{TraceInequality1} is strict, then the infimum is realized by a metric $g \in [g_0]$ that has constant mass.
\end{theoremA}

Here, $\mathrm{Tr}_\zeta(L_g^{-1})$ denotes the zeta-regularized trace of the inverse of $L_g$, which we define by $\mathrm{Tr}_\zeta(L_g^{-1}) := \mathrm{f.p.}_{s=1}\zeta_{L_g}(s)$, the finite part of the zeta function of $L_g$ at $s=1$. The motivation to call this quantity a trace comes from the observation that by the usual definition $\zeta_{L_g}(s) = \sum_{\lambda_j \neq 0} \lambda_j^{-s}$ for $\mathrm{Re}(s)$ large (where $\lambda_j$ are the eigenvalues of $L_g$), the value $\zeta_{L_g}(1)$ is formally the sum of the eigenvalues of $L_g^{-1}$. For general values of $m$ in even dimensions $n$, the zeta function of the $m$-th GJMS operator has a pole at $s=1$ so one needs to subtract this singularity to make $\mathrm{Tr}_\zeta(L_g^{-1})$ well-defined. However, it turns out that in the subcritical case $2m = n-2$ we consider, the zeta function is regular at $s=1$, so that $\mathrm{Tr}_\zeta(L_g^{-1}) = \zeta_{L_g}(1)$.

Before we comment on the notion of {\em mass}, we mention that Thm.~A above should be compared to the following result regarding the critical GJMS operator \cite{OkikioluGAFA}.

\begin{theoremnothing}[Okikiolu]
  Let $(X, g_0)$ be a Riemannian manifold of even dimension $n$ and let $L_{g_0} = L_{g_0}^{(n/2)}$ the critical GJMS operator. Suppose that $\ker L_{g_0} = \{\text{\normalfont constants}\}$. Then we have
\begin{equation} \label{TraceInequality2}
 \inf \mathrm{Tr}_\zeta(L_g^{-1}) \leq \mathrm{Tr}_\zeta(L_{g_{\mathrm{std}}}^{-1}),
\end{equation}
where $L_{g_{\mathrm{std}}}$ denotes critical GJMS operator on the standard sphere $(S^n, g_{\mathrm{std}})$ and the supremum is taken over all metrics $g$ in the conformal class of $g_0$ with volume equal to $\omega_n$. Moreover, if the inequality \eqref{TraceInequality2} is strict, then the infimum is realized by a metric $g \in [g_0]$ that has constant mass.
\end{theoremnothing}

We remark that the critical GJMS operator $L_g^{(n/2)}$ does not have a constant term, and that its kernel consists of the constant functions seems to be the generic situation. Similarly, the generic situation in the subcritical case seems to be $\ker L_g^{(n/2-1)} = 0$, which is precisely the assumption from Thm.~A. 

Notice that in comparison with Thm.~A, the supremum has been replaced by an infimum and the standard sphere now {\em maximizes} the quantity $\inf \mathrm{Tr}_\zeta(L_g^{-1})$. In fact, as $g$ varies over conformal metrics with fixed volume, $\mathrm{Tr}_\zeta(L_g^{-1})$ is unbounded below in the case of the critical GJMS operator while it is unbounded above in the subcritical case. If analogues of the above theorems hold in the case $n-2m = 4, 6,\dots$, we expect this alternating behavior to continue.

\medskip

The {\em mass} of a GJMS operator $L_g$ (which can also be defined for more general elliptic operators) is the function on $X$ defined by $\m_g(x) = \mathrm{f.p.}_{s=1}\zeta_{L_g}(s, x)$, where $\zeta_{L_g}(s, x)$ denotes the {\em local zeta function} of $L_g$. Again, in our subcritical case $2m=n-2$, we have $\m_g(x) = \zeta_{L_g}(1, x)$, because the local zeta function turns out to be regular at $s=1$.

In odd dimensions, the local zeta function is regular at $s=1$ for GJMS operators $L_g$ of any order $2m$, so one always has $\m_g(x) = \zeta_{L_g}(1, x)$. It has been shown by the author that in odd dimensions, the mass transforms very nicely under a conformal change $h = e^{2 \varphi} g$, namely
\begin{equation} \label{MassTransformationOdd}
  \m_h(x) = e^{(2m-n)\varphi(x)} \m_g(x),
\end{equation}
provided that $\ker L_g = 0$  \cite{LudewigMassInvariance}. This is not true in even dimensions, where in order to calculate $\m_h(x)$ from $\m_g(x)$, also the first $n-2m$ derivatives of $\varphi$ at $x$ are needed. However, we show below that in the case $2m = n-2$, $n$ even, one define the {\em normalized mass} $\m_g^{\mathrm{nor}}$ by
\begin{equation} \label{MassCorrectionIntro}
  \m_g^{\mathrm{nor}} := \m_g + b_n \mathrm{scal}_g,
\end{equation}
to obtain a quantity which transforms under a conformal change exactly by formula \eqref{MassTransformationOdd} (here $b_n$ is a dimensional constant, explicitly given in \eqref{DefNormalizedMass} below). It is natural to wonder if also in the case $n-2m = 4, 6, \dots$, one can modify $\m_g$ by a curvature term to obtain a quantity that transforms with the formula \eqref{MassTransformationOdd}. It seems intriguing to think that these corrections might be given by higher $Q$-curvatures. 

\medskip

The normalized mass defined in \eqref{MassCorrectionIntro} is used in the following theorem, which gives a sufficient condition for the inequality \eqref{TraceInequality1} to be strict.

\begin{theoremB}
  Let $(X, g_0)$ be a compact Riemannian manifold of even dimension $n \geq 4$ and let $\m_{g_0}^{\mathrm{nor}}$ be the normalized mass of the GJMS operator of order $2m=n-2$. Then if $\m_{g_0}^{\mathrm{nor}}(x_0)>0$ for some $x_0 \in X$, the equality \eqref{TraceInequality1} is strict. 
  \end{theoremB}
  
Using positive mass theorems, one can obtain geometric conditions for Thm.~B to hold.  
For example, suppose that $(X, g)$ is a four-dimensional manifold with positive Yamabe invariant. Then $L^{(n/2-1)}_{g_0} = Y_{g_0}$, the Yamabe operator, and we can use the positive mass theorem of Ammann and Humbert \cite[Section~3]{AmmannHumbert} to conclude that $\m_{g_0}^{\mathrm{nor}}(x)>0$ for each $x \in X$ unless $(X, {g_0})$ is conformally equivalent to the standard sphere (in which case the normalized mass is identically zero, cf.\ Lemma~\ref{LemmaMassOnStandardSphere} below). 

Similarly, if $(X, g)$ is a six-dimensional manifold with positive Yamabe invariant and (semi-)positive $Q$-curvature, then $L_{g_0}^{(n/2-1)} = L^{(2)}_{g_0}$, the Paneitz operator, and we can use the corresponding positive mass theorem (see \cite{HumbertRaulot} and \cite[Prop~2.9]{GurskyMalchiodi}) to conclude that $\m_{g_0}^{\mathrm{nor}}(x) >0$ for each $x \in M$ unless $(X, {g_0})$ is conformally equivalent to the standard sphere.

In any even dimension, we can at least say that when $(X, g)$ is conformally equivalent to real projective space $\R P^n$, the normalized mass $\m_{g_0}^{\mathrm{nor}}(x)$ of the GJMS operator of order $n-2$ is positive for every $x \in X$ (cf.\ Thm.~6.9 in \cite{LudewigMassInvariance}).

In all these cases, we obtain that the inequality \eqref{TraceInequality1} is strict and that the supremum is attained by a metric of constant mass. Sadly, our results to not allow to make these conclusions in the case that the mass is everywhere non-positive. On the other hand, little is known about the mass of manifolds with negative Yamabe constant. However, there are negative-mass theorems by Okikiolu concerning the critical GJMS operator on surfaces of positive genus \cite{OkikioluNegative1}, \cite{OkikioluNegative2}. 

\medskip

Using the formulas for the conformal change of the mass obtained below in the special case of the sphere, we can compute the zeta-regularized trace of the subcritical GJMS operator of any conformal metric on $S^n$ in terms of the conformal factor.

\begin{theoremC}
Let $g = u^{\frac{4}{n-2}} g_{\mathrm{std}}$ be a metric on $S^n$ in the conformal class of the standard sphere. Then the trace of the subcritical GJMS operator with respect to the metric $g = u^{\frac{4}{n-2}} g_{\mathrm{std}}$ on $S^n$ is
  \begin{equation} \label{TraceChange}
  \mathrm{Tr}_{\zeta}(L_g^{-1}) = - c_n\left(\int_{S^n} u \Delta u  \, \dd V_{g_{\mathrm{std}}}+ \frac{n(n-2)}{4} \int_{S^n} u^2 \, \dd V_{g_{\mathrm{std}}}\right),
\end{equation}
where $c_n$ is a dimensional constant, given explicitly in \eqref{QandCn} below.
\end{theoremC}

Using this formula, it is easy to see that our Thm.~A implies the standard Sobolev inequality on $S^n$
\begin{equation*}
  \int_{S^n} u \Delta_{g_{\mathrm{std}}} u \,\dd V_{g_{\mathrm{std}}} + \frac{n(n-2)}{4} \int_{S^n} u^2 \,\dd V_{g_{\mathrm{std}}} \geq \frac{n(n-2)}{4}  \omega_n^{2/n} \|u\|_p^2,
\end{equation*}
just as Okikiolu's theorem above implies the sharp logarithmic Hardy-Littlewood Sobolev inequality on $S^n$ using Morpurgo's formula \cite[Thm.~1]{Morpurgo}, which is the analog of \eqref{TraceChange} for the critical GJMS operator (cf.\ \cite{OkikioluGAFA}). Compare also to the results in Section~5 of \cite{MorpurgoSharp}.

\medskip

The structure of this paper is as follows. In Section~2, we discuss the behavior of the mass of the subcritical GJMS operators under a conformal change and introduce the normalized mass. In Section~3, we will introduce a functional, which we name mass functional, that turns out to be very useful to study the variational problem of the trace. We will see that this functional is closely related to the Yamabe functional, so that the task of finding a metric of constant mass can be solved in way similar to the solution of the Yamabe problem. In this section, we prove that a metric of constant mass exists in case that the inequality \eqref{TraceInequality1} is strict, as well as Thm.~C. In Section~4, we finish the proof of Thm.~A and Thm.~B by constructing suitable test functions for the mass functional.

\medskip

\textbf{Acknowledgements.} The author would like to thank the Max Planck Institute for Mathematics in Bonn for its hospitality and financial support.

\section{The Mass of GJMS Operators}

Let $(X, g)$ be an $n$-dimensional Riemannian manifold. It is well-known that under a conformal change $h := e^{2 \varphi} g$, the Laplace-Beltrami operator $\Delta_g$ transforms according to the formula
\begin{equation} \label{LaplacianTransformation}
  \Delta_h f = e^{-2 \varphi} \bigl(\Delta_g f - (n-2) \langle d\varphi, d f\rangle \bigr),
\end{equation}
while if 
\begin{equation*}
  Y_g f = \Delta_g f + a_n \mathrm{scal}_g f, ~~~~~~~\text{with}~~~~~~~~ a_n = \frac{n-2}{4(n-1)}
\end{equation*}
denotes the {\em Yamabe operator}, $Y_g$ has the simpler transformation formula 
\begin{equation} \label{YamabeTransformation1}
Y_hf = e^{-\frac{n+2}{2}\varphi} Y_g\bigl(e^{\frac{n-2}{2}\varphi} f\bigr).
\end{equation}
  This raises the question whether one could also add lower order terms to powers $\Delta_g^m$ of the Laplacian to make them obey a transformation law similar to \eqref{YamabeTransformation1} under a conformal change. In their seminal paper \cite{GJMS1}, Graham, Jenne, Mason and Sparling answered this question positively by constructing a natural family of such differential operators $L_g^{(m)}$, which are nowadays called GJMS operators. They exist on general Riemannian manifolds in the case that $1 \leq m \leq \frac{n}{2}$ if $n$ is even and for $m$ arbitrary if $n$ is odd and satisfy the transformation law
\begin{equation*}
   L_h^{(m)} f = e^{-\frac{n+2m}{2}\varphi} L_g^{(m)}\bigl(e^{\frac{n-2m}{2}\varphi} f\bigr),
\end{equation*}
similar to \eqref{YamabeTransformation1}.  The GJMS operators have the form $L_g = \Delta_g^m + \text{lower order}$, where the lower order terms are quantities locally determined by the curvature of $(X, g)$. Explicit formulas become more and more complicated for increasing $m$ and have only been worked out for small $m$. For example, we have $L_g^{(1)} = Y_g$, the Yamabe operator; $L^{(2)}_g$ is usually called the Paneitz operator, see \cite{Paneitz}; explicit formulas for $m=3, 4$ can be found e.g.\ in \cite{Juhl2}. However, the GJMS operators have a recursive structure which was investigated by Juhl \cite{Juhl4}. 

\medskip

All GJMS operators are semi-bounded, elliptic differential operators, and hence one can consider their local spectral zeta function
\begin{equation} \label{LocalZeta}
  \zeta_{L_g}(s, x) = \sum_{\lambda_j \neq 0} \lambda_j^{-s} \phi_j(x)^2,
\end{equation}
where $\lambda_j$ runs over all non-zero eigenvalues of $L_g$ with a corresponding orthonormal system of eigenfunctions $\phi_j$ (this definition makes sense for $\mathrm{Re}(s)$ large and for other values of $s$, $\zeta_{L_g}(s, x)$ is defined by analytic continuation). The zeta-regularized trace of $L_g^{-1}$ is now defined by the formula
\begin{equation} \label{DefinitionZetaTrace}
  \mathrm{Tr}_\zeta(L_g^{-1}) := \mathrm{f.p.}_{s=1} \zeta_{L_g}(s) = \int_X \mathrm{f.p.}_{s=1} \zeta_L(s, x) \,\dd V_g(x).
\end{equation}
Notice that formally plugging $s=1$ in \eqref{LocalZeta}, the right hand side of \eqref{DefinitionZetaTrace} is formally the sum over the eigenvalues of $L_g^{-1}$ (except for the fact that we have to take the finite part as $\zeta_{L_g}(s ,x)$ might have a pole at $s=1$). In the case that $n$ is odd, this coincides with the Kontsevich-Vishik trace of $L_g^{-1}$ (see \cite[Section~7.3]{KontsevichVishik}). In the case that $n$ is even, then the residue $\mathrm{res}_{s=1} \zeta_{L_g}(s)$ is equal to the Wodzicki non-commutative residue \cite{Wodzicki}, while the Kontsevich-Vishik trace is not defined. In this case, $\mathrm{Tr}_\zeta(L_g^{-1})$ coincides with the quasi-trace functional $C_0(L_g^{-1}, L_g)$ discussed by  discussed by Grubb(see \cite[Section~2]{Grubb} and the references therein).

\medskip

The {\em mass} of a GJMS operator $L_g$ (and of other suitable elliptic differential operators) at $x \in X$ is defined as the finite part of the local zeta function at $s=1$,
\begin{equation} \label{DefinitionMass}
  \m_g(x) := \mathrm{f.p.}_{s=1}\zeta_{L_g}(s, x) = \frac{\dd}{\dd s}\Bigr|_{s=1} (s-1) \zeta_{L_g}(s, x)
\end{equation}
i.e.\ the constant term in its Laurent expansion at $s=1$. If $n$ is odd, then $\zeta_{L_g}$ is regular at $s=1$ so that $\m_g(x) := \zeta_{L_g}(1, x)$. In even dimensions, the local zeta function $\zeta_{L_g}(s, x)$ has a pole at $s=1$, with residue is given by the so-called {\em logarithmic singularity}, which was investigated e.g.\ in \cite{PongeLogarithmicSingularity} in great detail. However, in the subcritical case $2m=n-2$, this logarithmic singularity is zero (see Thm.~7.5 in \cite{PongeLogarithmicSingularity}) so that also in this case, $\m_g(x) := \zeta_{L_g}(1, x)$.

\begin{remark}
The mass of $L_g$ can also be defined as the constant term in the asymptotic expansion of its Green's function near the diagonal. This definition was used e.g.\ in \cite{AmmannHumbert}, \cite{Habermann}, \cite{HumbertRaulot}, \cite{HermannHumbert} and other places. It is shown in \cite{LudewigMassInvariance} (cf.~Thm.~5.4) and certainly has been realized before by other people, that this notion coincides with the definition above.
\end{remark}

\begin{remark}
The reason for calling the number $\m_g(x)$ {\em mass} is that in the case of the Yamabe operator, this is related to the ADM mass of asymptotically flat manifolds, a concept from mathematical physics. Therefore, from a physicist's point of view, positivity results can be expected in presence positive energy. See Section~6 in \cite{LudewigMassInvariance} for a discussion of this, as well as the references therein.
\end{remark}

Let us now discuss how the mass changes under conformal transformations. As remarked in the introduction, if $n$ is odd and $\ker L_g = 0$, the mass of $L_g^{(m)}$ at $x\in X$ transforms according to the simple formula \eqref{MassTransformationOdd} under a conformal change $h = e^{2\varphi} g$. In even dimensions $n$, this is not true, but we can at least characterize the {\em infinitesimal} behavior of the mass under a conformal change. Namely, if we set $g_t := e^{2t\varphi}$ for $\varphi \in C^\infty(X)$, we have (still under the assumption $\ker L_g = 0$)
  \begin{equation} \label{ConformalVariationMass}
    \frac{\dd}{\dd t} \mathfrak{m}_{g_t}(x) = (2m-n) \varphi(x) \mathfrak{m}_{g_t}(x) + 2m (Q_{g_t} \varphi) (x),
  \end{equation}
where $Q_g$ is a certain differential operator of the form $c \Delta^{n/2-m}_g + \text{lower order}$ \cite[Thm.~7.1]{LudewigMassInvariance}. These operators $Q_g$ are formed, following a complicated recipe, out of the derivatives of heat kernel coefficients of the operator $L_g$ and hence theoretically can be explicitly computed from the formula given in the proof of Lemma~7.6 in \cite{LudewigMassInvariance}. In practice, however, as $n-2m$ increases, one needs more and more knowledge of the heat kernel coefficients of the operators in question, which are generally very hard to compute. For $n-2m = 2$, the formula is still manageable and the operator $Q_g$ is given as follows.

\begin{lemma} \label{LemmaExplicitQ}
  If $n-2m = 2$, the operator $Q_g$ from \eqref{ConformalVariationMass} is given explicitly by
  \begin{equation} \label{QandCn}
    Q_g = - c_n \Delta_g, ~~~~~~~~\text{where}~~~~~~~~ c_n = \frac{(n-2)}{6 (4 \pi)^{n/2} \left(\frac{n}{2}\right)!}.
  \end{equation}
\end{lemma}

\begin{proof}
To explicitly calculate $Q_g$ brute force by the formula in the proof of Lemma~7.6 in \cite{LudewigMassInvariance}, one needs the knowledge of the heat kernel coefficients $\Phi_0$ and $\Phi_1$ on the diagonal, along with the derivatives of $\Phi_0$. This is not too involved. One can be a bit more clever, however: It is not hard to show that all these $Q$ operators are self-adjoint and have the heat kernel coefficient $\Phi_{n/2-m}$ of $L_g$, evaluated at the diagonal, as constant term (which equals the logarithmic singularity of \cite{PongeLogarithmicSingularity} up to a dimensional factor). By Thm.~7.5 in \cite{PongeLogarithmicSingularity}, the logarithmic singularity is zero in the case $n-2m=2$, so that $Q_g$ is a constant multiple of $\Delta_g$. The constant depends only on $m$ and $n$ and can be explicitly calculated using the formulas from the Lemma~7.6 mentioned above. In particular, for $2m=n-2$, one calculates this constant to be $- c_n$.\footnote{The proof of Lemma~7.6 of \cite{LudewigMassInvariance} is erroneous in the printed version. For the correct formulas, appeal to the arxiv version.}
\end{proof}

This explicit formula for $Q_g$ in the case $2m=n-2$ now allows us to integrate \eqref{ConformalVariationMass} in order to obtain a non-infinitesimal version of the equation in this case.

\begin{proposition} \label{PropMassTransformation}
If $2m = n-2$, we have
\begin{equation} \label{MassTransformationEven}
\mathfrak{m}_{u^{\frac{4}{n-2}} g} = -u^{-{\frac{n+2}{n-2}}} P_g u,
\end{equation}
$u \in C^\infty(X)$ with $u>0$, where the operator $P_g$ is defined by $P_g f := c_n \Delta_g f - \mathfrak{m}_g f$. Under a conformal change, $P_g$ from above transforms according to
\begin{equation} \label{PTransformation}
  P_{e^{2\varphi}g} f = e^{-\frac{n+2}{2} \varphi} P_g (e^{\frac{n-2}{2} \varphi} f),
\end{equation}
just as the Yamabe operator.
\end{proposition}

\begin{proof}
Set $g_t:= e^{2t \varphi} g$ as above. Then integrating the relation \eqref{ConformalVariationMass}, we have
\begin{equation*}
  \m_{g_1} = e^{-2 \varphi} \left( \m_g + (n-2) \int_0^1  e^{2 t \varphi} Q_{g_t} \varphi \dd t\right).
\end{equation*}
Because of Lemma~\ref{LemmaExplicitQ} and the formula \eqref{LaplacianTransformation} for the behavior of the Laplacian under a conformal change, $Q_g$ transforms according to
\begin{equation*}
  Q_{g_t} f = - c_n \Delta_{g_t} f = - c_n e^{-2 t \varphi}\Bigl( \Delta_g f - (n-2) t \langle d\varphi, d f \rangle \Bigr).
\end{equation*}
Hence
\begin{equation*}
\begin{aligned}
(n-2)\int_0^1 e^{2 t \varphi}  Q_{g_t} \varphi \dd t &= - (n-2) c_n \int_0^1 \bigl( \Delta_g\varphi - (n-2) t |d \varphi|^2 \bigr) \dd t \\
&= - (n-2) c_n\left( \Delta_g\varphi - \frac{n-2}{2} |d\varphi|^2\right),
\end{aligned}
\end{equation*}
so that
\begin{equation} \label{ConformalChangeMWithPhi}
 \m_{e^{2 \varphi} g} = e^{-2 \varphi} \left( \m_g - (n-2) c_n \left(\Delta \varphi - \frac{n-2}{2} |d \varphi|^2 \right) \right)
\end{equation}
Writing $e^{2 \varphi} = u^{\frac{4}{n-2}}$, we have
\begin{equation*}
   \Delta_g\varphi - \frac{n-2}{2} |d\varphi|^2 = \frac{1}{n-2} \frac{\Delta_g u}{u},
\end{equation*}
which implies \eqref{MassTransformationEven}.

To see how $P_g$ transforms under a conformal change, calculate using \eqref{LaplacianTransformation} again 
  \begin{equation*}
  \begin{aligned}
    P_{e^{2 \varphi} g}f  &= c_n \Delta_{e^{2\varphi}g} f - \m_{e^{2\varphi}g} f \\
    &= e^{-2 \varphi} \left( c_n \Bigl(\Delta_g f - (n-2) \< d \varphi, df\>\Bigr) - \m_g f + (n-2) c_n \left(\Delta \varphi - \frac{n-2}{2} |d \varphi|^2 \right)f \right)\\
    &= e^{-2 \varphi} \left( c_n e^{-\frac{n-2}{2} \varphi} \Delta_g (e^{-\frac{n-2}{2} \varphi} f)  - \m_g f  \right)\\
    &= e^{-\frac{n+2}{2} \varphi} P_g (e^{-\frac{n-2}{2} \varphi} f),
  \end{aligned}
  \end{equation*}
  where we also used \eqref{ConformalChangeMWithPhi} and the product rule for the Laplacian.
\end{proof}

\begin{corollary}
  Defining
  \begin{equation} \label{DefNormalizedMass}
    \m^{\mathrm{nor}}_g := \m_g + b_n \mathrm{scal}_g ~~~~~~~~\text{with}~~~~~~~~ b_n =a_n c_n = \frac{(n-2)^2}{24(n-1)(4\pi)^{n/2} \left(\frac{n}{2}\right)!}
    \end{equation}
  we obtain that this {\em normalized mass} $\m^{\mathrm{nor}}_g$ transforms under a conformal change according to
  \begin{equation} \label{TransformationMNor}
    \m^{\mathrm{nor}}_{e^{2 \varphi} g} = e^{-2 \varphi}  \m^{\mathrm{nor}}_{g},
  \end{equation}
  similar to formula \eqref{MassTransformationOdd} in odd dimensions.
\end{corollary}

\begin{proof}
 It is well known \cite[(1.8)]{Yamabe} that the scalar curvature of $u^{\frac{4}{n-2}} g$ is given by
 \begin{equation*}
   \alpha_n\mathrm{scal}_{u^{\frac{4}{n-2}} g} = u^{-\frac{n+2}{n-2}} Y_g u = u^{-\frac{n+2}{n-2}} (\Delta_g u + a_n \mathrm{scal}_g u).
 \end{equation*}
 Setting $u = e^{\frac{n-2}{2} \varphi}$, i.e.\ $e^{2\varphi} = u^{\frac{4}{n-2}}$, we have using \eqref{MassTransformationEven} and \eqref{PTransformation}
  \begin{equation*}
  \begin{aligned}
      \m^{\mathrm{nor}}_{u^{\frac{4}{n-2}}g} &= \m_{u^{\frac{4}{n-2}}g} + b_n \mathrm{scal}_{u^{\frac{4}{n-2}}g}\\
      &=  -u^{-{\frac{n+2}{n-2}}} P_g u + b_n u^{-{\frac{n+2}{n-2}}}\left( a_n^{-1} \Delta_g u + \mathrm{scal}_g u \right)\\
      &= u^{-{\frac{n+2}{n-2}}}\bigl( \m_g u - (c_n - b_na_n^{-1})\Delta_g u +  b_n \mathrm{scal}_g u\bigr) = u^{-\frac{4}{n-2}} \m^{\mathrm{nor}}_{g}.
  \end{aligned} 
  \end{equation*}
This finishes the proof.
\end{proof}

\newpage

The following lemma suggests that this is the "right" definition of the normalized mass.

\begin{lemma} \label{LemmaMassOnStandardSphere}
  On the standard sphere with the round metric $g_{\mathrm{std}}$, we have $\m^{\mathrm{nor}}_{g_{\mathrm{std}}} \equiv 0$.
\end{lemma}

\begin{proof}
Let $x_0 \in S^n$. Let $g$ be a metric on $S^n$ that is flat near $x_0$ and satisfies $g_{\mathrm{std}} = e^{2\varphi} g$ for some function $\varphi \in C^\infty(M)$.
Now by \eqref{TransformationMNor} and the fact that $\mathrm{scal}_g(x_0) = 0$, we have $\m_{g_{\mathrm{std}}}^{\mathrm{nor}} = e^{-2\varphi(x_0)}\m_g^{\mathrm{nor}}(x_0) = e^{-2\varphi(x_0)}\m_g(x_0)$. On the other hand, $(S^n, g)$ is simply connected, locally conformally flat and flat near $x_0$. By \cite[Thm.~6.9]{LudewigMassInvariance}, this implies that $\m_g(x_0) = 0$.
\end{proof}

It would be nice to use a similar strategy as in the proof of Prop.~\ref{PropMassTransformation} also in the cases that $n-2m = 4, 6, \dots$ etc. to obtain a non-infinitesimal version of \eqref{ConformalVariationMass}. However, in these cases, we could not obtain an explicit formula such as \eqref{LemmaExplicitQ} yet; already for $m=2$, the formula for $Q_g$ given in \cite{LudewigMassInvariance} becomes incredibly complicated and involves the second heat kernel coefficient, second derivatives of the first heat kernel coefficient and fourth derivatives of the index zero heat kernel coefficient.

\section{The Mass Functional}

Let $(X, g)$ be a compact Riemannian manifold of dimension $n \geq 4$ and let $L_g = L_g^{(m)}$ be a GJMS operator of order $2m$. To study the variational problem for the trace of $L_g$, it is natural to consider the {\em mass functional}
\begin{equation} \label{DefMassFunctional}
  \mathcal{M}^{(m)}(X, g) := \frac{\int_X \m_g^{(m)}(x) \dd V_g(x)}{\mathrm{vol}(X, g)^{\frac{2m}{n}}} 
\end{equation}
In odd dimensions, dividing by the volume to the power of $\frac{2m}{n}$ makes the functional scale invariant by \eqref{MassTransformationOdd}, and we will see that the same is true in the case that $n$ is even, $2m = n-2$ and $\ker L_g = 0$. Notice furthermore that by the definition \eqref{DefinitionMass} of the mass, the functional $\mathcal{M}^{(m)}$ is related to the zeta-regularized trace via
\begin{equation} \label{TraceMassFunctional}
\mathrm{Tr}_{\zeta}(L^{-1}_g) = \mathcal{M}^{(m)}(X, g) \cdot \mathrm{vol}(X, g)^{\frac{2m}{n}},
\end{equation}
so that the mass functional seems to be a suitable tool to vary the trace among metrics of fixed volume.

We now restrict to the case that $2m = n-2$. Fixing a metric $g$, we define $\mathcal{M}_g(u) := \mathcal{M}^{(n-2)}(X, u^{\frac{4}{n-2}}g)$ for $u \in C^\infty(X)$. Then by definition, we have
\begin{equation*}
 \sup_{g \in [g_0]} \mathcal{M}^{(n-2)}(X, g) = \sup_{\substack{u \in C^\infty(X)\\u > 0}} \mathcal{M}_{g}(u),
\end{equation*} 
for any metric $g \in [g_0]$. From Prop.~\ref{PropMassTransformation}, we now obtain the following more explicit formula for the mass functional. Using Lemma~\ref{LemmaMassOnStandardSphere}, this proposition directly implies Thm.~C, with a view on \eqref{TraceMassFunctional}.

\begin{proposition} \label{PropMassDecomposition}
  Let $(X, g)$ be a compact Riemannian manifold of even dimension $n \geq 4$ and let $L_g$ be the GJMS operator of order $2m = n-2$, with associated mass functional $\mathcal{M}_g$. Assume $\ker L_g = 0$. Then we have
  \begin{equation} \label{MassFunctionalDecomposition}
    \mathcal{M}_g(u) = - \frac{\int_X u P_g u \,\dd V_g}{\|u\|_p^2} = \frac{\int_X \m^{\mathrm{nor}}_g u^2 \, \dd V_g}{\|u\|_p^2} - b_n \mathcal{Y}_g(u), 
  \end{equation}
  where $p = \frac{2n}{n-2}$, $\|u\|_p$ denotes the $L^p$ norm with respect to the metric $g$ and
  \begin{equation} \label{RelativeYamabeFunctional}
    \mathcal{Y}_g(u) = \frac{1}{a_n} \frac{\int_X u Y_g u \, \dd V_g}{\|u\|_p^2}
  \end{equation}
  is the Yamabe functional.
\end{proposition}  

\begin{proof}
By Prop.~\ref{PropMassTransformation},
\begin{equation*}
  \mathcal{M}_g(u) = \frac{\int_X \m_{u^{\frac{4}{n-2}}g}(x) \dd V_{u^{\frac{4}{n-2}}g}(x)}{\mathrm{vol}(X, u^{\frac{4}{n-2}}g)^{\frac{2m}{n}}} 
  = -\frac{\int_X u^{-\frac{n+2}{n-2}} (P_g u) u^{\frac{2n}{n-2}}\dd V_g(x)}{\left(\int_X u^{\frac{2n}{n-2}}\dd V_g\right)^{\frac{2m}{n}}} = - \frac{\int_X u P_g u \,\dd V_g}{\|u\|_p^2}.
\end{equation*}
  Furthermore, by \eqref{DefNormalizedMass}, we have
  \begin{equation*}
    P_g = c_n \Delta_g - \m_g = c_n\Delta_g + b_n \mathrm{scal}_g - \m^{\mathrm{nor}}_g 
    = b_n a_n^{-1} Y_g - \m^{\mathrm{nor}}_g.
  \end{equation*}
  This proves the proposition.
\end{proof}  

Using that by \eqref{MassFunctionalDecomposition}, the mass functional can be written as an energy functional associated to a second order elliptic differential operator, it is a standard observation that we have
\begin{equation*}
 \sup_{g \in [g_0]} \mathcal{M}^{(n-2)}(X, g) = \sup_{u \in C^\infty(X)} \mathcal{M}_{g}(u),
\end{equation*} 
for any metric $g \in [g_0]$. That is, we do not necessarily have to take positive functions $u$ as our test functions. 

\medskip

By the relation \eqref{TraceMassFunctional}, the following proposition proves one half of Thm.~A; the other half is proven by Prop.~\ref{PropTestFunctionEstimate1} below by constructing a test function for the mass functional.

\begin{proposition} \label{PropMinimizerExists}
  Let $(X, g_0)$ be a closed connected Riemannian manifold of even dimension $n \geq 4$ and let $L_{g_0}$ be the GJMS operator of order $2m=n-2$. Assume $\ker L_{g_0} = 0$. Then if 
  \begin{equation*}
     \sup_{g \in [g_0]} \mathcal{M}^{(n-2)}(X, g) > \mathcal{M}^{(n-2)}(S^n, g_{\mathrm{std}}),
  \end{equation*}
 the supremum is attained by a metric $g \in [g_0]$, and this metric has constant mass. 
\end{proposition}

\begin{proof}
By formula \eqref{MassFunctionalDecomposition}, differentiating the functional $\mathcal{M}_g(u)$ yields
\begin{equation*}
  \frac{\dd}{\dd s}\Bigr|_{s=0} \mathcal{M}_g(u+sv) = - \frac{2}{\|u\|_p^2} \int_X \left( P_g u - \frac{\mathcal{M}_g(u)}{\|u\|^{p-2}_p} |u|^{p-2}u\right) v\, \dd V_g.
\end{equation*}
Therefore, since any positive smooth minimizer $u$ is necessarily a critical point of $\mathcal{M}_g$, such a minimizer necessarily satisfies the partial differential equation
\begin{equation} \label{CriticalPDE}
  P_g u = \Lambda u^{p-1} ,
\end{equation}
for some $\Lambda \in \R$, which implies together with \eqref{MassTransformationEven} that the metric $u^{\frac{4}{n-2}}g$ has constant mass.

We now discuss the problem of finding a minimizer. For a general operator $P_g$ of the form $P_g = c\Delta_g + f$ for $f \in C^\infty(X)$, $c \in \R$, consider the functional
\begin{equation*}
  \mathcal{P}_g(u) := \frac{\int_X u P_g u \,\dd V_g}{\|u\|_p^2}
\end{equation*}
From the combined efforts of Yamabe \cite{Yamabe}, Trudinger \cite{Trudinger} and Aubin \cite{Aubin, Aubin2}, we know how to construct a minimizer of such a functional: a smooth and positive minimizer exists in the case that
\begin{equation*}
  \inf_{u \in C^\infty(X)} \frac{\int_X u P_g u \,\dd V_g}{\|u\|_p^2} < c a_n \mathcal{Y}(S^n, g_{\mathrm{std}}),
\end{equation*}
where $\mathcal{Y}(S^n, g_{\mathrm{std}}) \equiv n(n-1) \omega_n^{2/n}$ is the Yamabe constant of the standard sphere. This result is usually formulated in the case that $c=1$ and $f=a_n\mathrm{scal}$ in which $P_g$ is the Yamabe operator, but following e.g.\ the proof in Section~4 of \cite{LeeParker} gives the same result in this more general setting.

In our case, $c= c_n = b_n a_n^{-1}$, $f = -\m_g$. Then $\mathcal{M}_g(u) = - \mathcal{P} _g(u)$ by \eqref{MassFunctionalDecomposition}, so that a smooth, positive maximizer $u$ of the mass functional exists, provided
\begin{equation*}
  \sup_{g \in [g_0]} \mathcal{M}_g(u) > - b_n \mathcal{Y}(S^n, g_{\mathrm{std}}).
\end{equation*}

Finally by Lemma~\ref{LemmaMassOnStandardSphere}, we have $\m^{\mathrm{nor}}_{g_{\mathrm{std}}} \equiv 0$ on $S^n$ with the standard metric $g_{\mathrm{std}}$. Hence by \eqref{MassFunctionalDecomposition}, we have $-b_n \mathcal{Y}(S^n, g_{\mathrm{std}}) = \mathcal{M}^{(n-2)}(S^n, g_{\mathrm{std}})$, which finishes the proof.
\end{proof}

\begin{proposition} \label{PropEquality}
  Within the conformal class of the standard metric on $S^n$, the mass functional is maximized precisely at the standard metric and its images under conformal diffeomorphisms.
\end{proposition}

\begin{proof}
By Lemma~\ref{LemmaMassOnStandardSphere}, we have $\m^{\mathrm{nor}}_{g_{\mathrm{std}}} \equiv 0$ on $S^n$ with the standard metric $g_{\mathrm{std}}$ and by the conformal transformation law \eqref{TransformationMNor} of the normalized mass, the same is true for any metric $g \in [g_{\mathrm{std}}]$. Therefore, from \eqref{MassFunctionalDecomposition}, we obtain
\begin{equation} \label{MassFunctionalSphere}
-b_n \mathcal{Y}(S^n, g) = \mathcal{M}^{(n-2)}(S^n, g)
\end{equation}
for all $g \in [g_{\mathrm{std}}]$. Hence the proposition follows directly from the corresponding fact for the Yamabe functional (cf.\ \cite[Thm.~3.2]{LeeParker}.
\end{proof}

\section{A Test Function Estimate}

In this section, we finish the proof of Thm.~A as well as Thm.~B, by constructing suitable test functions for the mass functional.

\medskip

With a view on \eqref{TraceMassFunctional}, Thm.~A follows from Prop.~\ref{PropMinimizerExists} together with the following assertion.

\begin{proposition} \label{PropTestFunctionEstimate1}
Let $(X, g_0)$ be a closed connected Riemannian manifold of even dimension $n \geq 4$. Suppose that $\ker L_{g_0} = 0$, where $L_{g_0}$ is the GJMS operator of order $2m=n-2$. Then we have
\begin{equation*}
\sup_{g \in [g_0]} \mathcal{M}^{(n-2)}(X, g) \geq \mathcal{M}^{(n-2)}(S^n, g_{\mathrm{std}})
\end{equation*}
for the corresponding mass functionals.
\end{proposition}

\begin{proof}
Let $g \in [g_0]$ be a metric such that in Riemannian normal coordinates around $x_0$, we have $\det(g_{ij}) \equiv 1$. Such a metric, called {\em conformal normal coordinates}, is well known to exist (cf.\ \cite[Section~5]{LeeParker}, \cite{Guenther}). Moreover, we can choose the conformal factor relating $g$ and  $g_0$ to be equal to one at $x_0$, so that  $\m_g^{\mathrm{nor}}(x_0) = \m_{g_0}^{\mathrm{nor}}(x_0)$ by the transformation law \eqref{TransformationMNor}. Another feature of conformal normal coordinates around $x_0$ is that one has
\begin{equation} \label{ScalarCurvatureEstimate}
  |\mathrm{scal}_g| \leq C r^2,
\end{equation}
for some $C>0$, where $r$ denotes the distance function from $x_0$ \cite[Thm.~5.1]{LeeParker}. These observations will simplify several calculations.

The proof now consists of finding a suitable family of test functions for the functional $\mathcal{M}_g(u)$. The construction of these will be very similar to the approach in \cite[Section~3]{LeeParker}.

For $\alpha >0$, define $u_\alpha \in C^\infty(\R^n)$ by
\begin{equation*}
  u_\alpha(x) := \left( \frac{|x|^2 + \alpha^2}{\alpha}\right)^{\frac{2-n}{2}}.
\end{equation*}
For a suitable value of $\alpha$, this is the conformal factor that relates the standard metric in $\R^n$ to the round metric of $S^n$ when the latter is pushed forward to $\R^n$ via the stereographic projection. One of the features of this family of functions $u_\alpha$ is that for $p=\frac{2n}{n-2}$, the $L^p$ norm $\|u_\alpha\|_p$ is independent of $\alpha$; in fact $\|u_\alpha\|_p^p \equiv 2^{-n}{\omega_n}$.

For $\varepsilon>0$ small, let furthermore $\eta$ be a smooth function on $\R^n$ which satisfies $\eta(x) = 1$ for $|x| < \varepsilon$ and $\eta(x) = 0$ for $|x| > 2\varepsilon$. Finally, given Riemannian normal coordinates $\textbf{x}$ around $x_0$, defined on $U \subset X$, we set on $U$
\begin{equation*}
  \psi_{\alpha} := \textbf{x}^* (\eta \cdot  u_\alpha)
\end{equation*}
and extend $\psi_{\alpha}$ by zero to a function on all of $X$. Here we choose $\varepsilon$ so small that $B_{2\varepsilon}(x_0) \subset U$, meaning that $\psi_\alpha \in C^\infty(X)$.

We need to estimate
\begin{equation} \label{ToEstimate}
  \int_X \psi_{\alpha} P_g \psi_{\alpha} \dd V_g = c_n \int_X |d \varphi_{\alpha}|^2 \dd V_g + b_n\int_X \mathrm{scal}_g \psi_{\alpha}^2 \dd V_g - \int_X \m_g^{\mathrm{nor}} \psi_{\alpha}^2 \dd V_g
\end{equation}
from below by $\|\psi_\alpha\|_p^2$.
Notice that in normal coordinates, we have $g^{rr} = 1$, hence since $\varphi_{\alpha}$ is a radial function, we have
\begin{equation} \label{Term1a}
\begin{aligned}
  \int_X |d \varphi_{\alpha}|^2 \dd V_g &= \int_{B_{2\varepsilon}(x_0)} g^{rr} (\partial_r \varphi_{\alpha})^2 \dd V_g 
  = \int_{\R^n} \bigl(\partial_r(\eta\cdot u_\alpha)\bigr)^2 \dd x \\
  &= \int_{\R^n} \Bigl( \eta^2 |du_\alpha|^2 + 2\eta u_\alpha (\partial_r\eta)(\partial_r u_\alpha) + (\partial_r \eta)^2 u_\alpha^2\Bigr) \dd x.
  \end{aligned}
\end{equation}
Here we also used that $\det(g_{ij}) \equiv 1$ on a neighborhood of $x_0$ so that $\mathbf{x}_* \dd V_g = \dd x$ on $U$. For the second term of \eqref{Term1a}, we have
\begin{equation*}
\begin{aligned}
  2\int_{\R^n} \eta u_\alpha (\partial_r\eta)(\partial_r u_\alpha) \dd x &\leq C \int_{B_{2\varepsilon}(0) \setminus B_{\varepsilon}(0)} \bigl|u_\alpha \cdot \partial_r u_\alpha\bigr| \dd x \\
  &= C \omega_{n-1}(n-2) \int_\varepsilon^{2\varepsilon} r^{n-1} \bigl|u_\alpha(r) \cdot \partial_r u_\alpha(r) \bigr|\dd r \\
  &\leq (n-2) \alpha^{n-2} \int_\varepsilon^{2\varepsilon} r^{n-1}(r^2 + \alpha^2)^{\frac{2-n}{2}} (r^2 + \alpha^2)^{-\frac{n}{2}} r \dd r\\
  &\leq (n-2) \alpha^{n-2} \int_\varepsilon^{2\varepsilon} r^{n-1} r^{2-n} r^{-n} r \dd r\\
   &\leq C \alpha^{n-2}
\end{aligned}
\end{equation*}
Here and in the following, $C$ denotes some positive constant (independent of $\alpha$), the exact value of which is unimportant and may change from line to line.
Similarly, we have
\begin{equation*}
\int_{\R^n}  (\partial_r \eta)^2 u_\alpha^2 \dd x \leq C \alpha^{n-2}.
\end{equation*}
Hence the latter two terms in \eqref{Term1a} can both be estimated by $C \alpha^{n-2}$ and we obtain
\begin{equation} \label{Term1}
   c_n\int_X |d \varphi_{\alpha}|^2 \dd V_g \leq c_n\|du_\alpha\|_2^2 + C \alpha^{n-2} = b_n\mathcal{Y}(S^n, g_{\mathrm{std}}) \|u_\alpha\|_p^2 + C \alpha^{n-2},
\end{equation}
where we used that we have $\|du_\alpha\|_2^2 = a_n\mathcal{Y}(S^n, g_{\mathrm{std}}) \|u_\alpha\|_p^2$ for any $\alpha>0$ (compare by \cite[Thm.~3.3]{LeeParker}).

To estimate the other terms of \eqref{ToEstimate}, we need the following calculus lemma (this can be found as Lemma~3.5 in \cite{LeeParker}; the proof is simple and we do not repeat it here).

\begin{lemma} \label{LemmaLeeParker}
  Suppose $k > -n$. Then as $\alpha \rightarrow 0$, the integral
  \begin{equation*}
    \int_0^\varepsilon u_\alpha(r)^2 r^{k+ n-1} \dd r
  \end{equation*}
  is bounded above and below by constant multiples of $\alpha^{k+2}$ if $n>k+4$, $\alpha^{k+2}\log(1/\alpha)$ if $n=k+4$ and $\alpha^{n-2}$ if $n< k+4$. 
\end{lemma}

Now we can estimate using Lemma~\ref{LemmaLeeParker} with $k=0$ 
\begin{equation} \label{Term2and3}
\begin{aligned}
\int_X \bigl(b_n\mathrm{scal}_g  -  \m_g^{\mathrm{nor}} \bigr) \psi_{\alpha}^2 \dd V_g &\leq C \int_{B_{2\varepsilon}(x_0)} \psi_{\alpha}^2 \dd V_g
\leq C \int_{B_{2\varepsilon}(0)} u_\alpha(x)^2 \dd x\\
&= C \omega_{n-1} \int_0^{2\varepsilon} u_{\alpha}(r)^2 r^{n-1} \dd r \\
&\leq \left.\begin{cases} C \alpha^2 & n \geq 6 \\ C \alpha^2 \log(\alpha) & n=4\end{cases} \right\} \leq C \alpha,
\end{aligned}
\end{equation}
where the third integral is over the $2\varepsilon$-ball in $\R^n$.
We also need an estimate of $\|\varphi_{\alpha}\|_p$ from below. Here, again because we are working in conformal normal coordinates, we have
\begin{equation*}
 \|\varphi_{\alpha}\|_p^p = \int_{B_{2\varepsilon}(x_0)} \varphi_{\alpha}(x)^p \dd V_g(x)  = \int_{B_{\varepsilon}(0)} \eta(x)^p u_\alpha(x)^p \dd x = \|\eta u_\alpha\|_p^p.
\end{equation*}
Now $(r^2+1)^{-n} \leq r^{-2n}$ for $r>0$ so that
\begin{equation*}
\begin{aligned}
  \|(1-\eta)u_\alpha\|_p^p &\leq \int_{\R^n \setminus B_{\varepsilon}(0)} u_\alpha(x)^p \dd x = \omega_{n-1} \int_\varepsilon^\infty \frac{\alpha^n r^{n-1}}{(r^2+\alpha^2)^n} \dd r\\
  &= \omega_{n-1} \int_{\varepsilon/\alpha}^\infty \frac{r^{n-1}}{(r^2+1)^n} \dd r \leq \omega_{n-1}  \int_{\varepsilon/\alpha}^\infty r^{-1-n} \dd r = \frac{\omega_{n-1}}{n\varepsilon^n} \alpha^n ,
\end{aligned}  
\end{equation*}
hence
\begin{equation} \label{EstimateFromBelow}
  \|\psi_{\alpha}\|_p^2  = \|\eta u_\alpha\|_p^2 = \bigl(\| u_\alpha\|_p^p - \|(1-\eta)u_\alpha\|^p\bigr)^{2/p} \geq \|u_\alpha\|_p^2 \left( 1 - C \alpha^n\right)^{2/p}.
\end{equation}
Plugging \eqref{Term1} and \eqref{Term2and3} into \eqref{ToEstimate} and using that $\|u_\alpha\|_p$ is independent from $\alpha$ then yields for $\alpha$ small
\begin{equation*}
\begin{aligned}
\int_X \psi_{\alpha} P_g \psi_{\alpha} \dd V_g &\leq \Bigl(b_n \mathcal{Y}(S^n, g_{\mathrm{std}}) + C\alpha^{n-2} + C \alpha\Bigr) \|u_\alpha\|_p^2\\
&\leq \Bigl(b_n \mathcal{Y}(S^n, g_{\mathrm{std}}) + C\alpha^{n-2} + C \alpha\Bigr)  \bigl( 1 - C \alpha^n \bigr)^{-2/p}  \|\psi_{\alpha}\|_p^2.
\end{aligned}
\end{equation*}
We obtain the desired estimate 
\begin{equation*}
- \sup_{\alpha>0} \mathcal{M}_g(\psi_{\alpha}) = \inf_{\alpha>0} -\mathcal{M}_g(\psi_{\alpha}) = \inf_{\alpha>0} \frac{ \int_X \psi_{\alpha} P_g \psi_{\alpha} \dd V_g}{\|\psi_{\alpha}\|_p^2}\leq b_n \mathcal{Y}(S^n, g_{\mathrm{std}}).
\end{equation*}
This finishes the proof with a view on \eqref{MassFunctionalSphere}.
\end{proof}

We now proceed to the proof of Thm.~B. We remark that with a view on \eqref{MassFunctionalDecomposition}, Thm.~B is obvious in the case that $\m^{\mathrm{nor}}(x)\geq 0$ for all $x \in M$, as then
\begin{equation*}
  \sup_{g \in [g_0]} \mathcal{M}^{(n-2)}(X, g) \geq -\inf_{g \in [g_0]} b_n\mathcal{Y}(X, g),
\end{equation*}
and it is well-known that the infimum over $\mathcal{Y}(X, g)$ is strictly smaller than $\mathcal{Y}(S^n, g_{\mathrm{std}})$ unless $(X, g)$ is conformal to the standard sphere (see Thm.~B and C in \cite{LeeParker}), so that Thm.~B follows in this case with a view on \eqref{MassFunctionalSphere}.

\begin{proof}[of Thm.~B]
We will use the same test function $\psi_\alpha$ as in the proof of Prop.~\ref{PropTestFunctionEstimate1} above. However, in the case that $\m_g^{\mathrm{nor}}(x_0) >0$, we can obtain a better result by estimating the second and the third term on the right hand side in \eqref{ToEstimate} separately. For the second term, we use \eqref{ScalarCurvatureEstimate} and Lemma~\ref{LemmaLeeParker} with $k=2$ to obtain
\begin{equation} \label{Term2}
\begin{aligned}
  b_n\int_X \mathrm{scal}_g \psi_{\alpha,}^2 \dd V_g &\leq C \int_{B_{2\varepsilon}(x_0)} r^2 \psi_{\alpha}^2 \dd V_g \leq C \int_0^{2\varepsilon} r^{n+1} u_\alpha(r)^2 \dd r \\
  &\leq 
  \begin{cases} C \alpha^4 & n \geq 8 \\ C\alpha^4 \log(1/\alpha) & n = 6 \\ C\alpha^2 & n = 4 \end{cases}.
  \end{aligned}
\end{equation}
Regarding the third term, notice that since $\m_g^{\mathrm{nor}}(x_0) >0$, the same is true for small neighborhoods of $x_0$, so we can assume that  $\varepsilon$ was chosen so small that
\begin{equation*}
  \inf_{x \in B_{2\varepsilon}(x_0)} \m_g^{\mathrm{nor}}(x) =: \boldsymbol{\mu} >0.
\end{equation*}
Then using Lemma~\ref{LemmaLeeParker} with $k=0$ for a lower estimate, we obtain
\begin{equation} \label{Term3}
\begin{aligned}
  - \int_X \m_g^{\mathrm{nor}} \psi_{\alpha}^2 \dd V_g &\leq - \boldsymbol{\mu} \int_X \psi_{\alpha}^2 \dd V_g \leq -\boldsymbol{\mu} \omega_{n-1} \int_0^\varepsilon u_\alpha(r)^2 r^{n-1} \dd r \\
  &\leq \begin{cases} -\boldsymbol{\mu} \delta \alpha^2 & n\geq 6 \\ -\boldsymbol{\mu} \delta \alpha^2 \log(1/\alpha) & n=4\end{cases}
\end{aligned}
\end{equation}
for some small $\delta>0$. 
Plugging the three individual estimates \eqref{Term1}, \eqref{Term2} and \eqref{Term3}  into \eqref{ToEstimate} and then using the estimate \eqref{EstimateFromBelow} on $\|\psi_\alpha\|_p^2$ (where we remember that $\|u_\alpha\|_p$ is independent of $\alpha$ and thus may be swallowed by the constants) gives in the case $n\geq 8$
\begin{equation*}
\begin{aligned}
\int_X \psi_{\alpha} P_g \psi_{\alpha} \dd V_g &\leq \frac{ b_n \mathcal{Y}(S^n, g_{\mathrm{std}})+ C \alpha^{n-2}+ C \alpha^4 -\boldsymbol{\mu} \delta \|u_1\|_p^{-2} \alpha^2}{(1 - C \alpha^n )^{2/p}}\|\psi_{\alpha}\|_p^2.
\end{aligned}
\end{equation*}
If $n=6$, we get 
\begin{equation*}
\begin{aligned}
\int_X \psi_{\alpha} P_g \psi_{\alpha} \dd V_g &\leq \frac{b_n \mathcal{Y}(S^n, g_{\mathrm{std}})  + C \alpha^{4}+ C \alpha^4\log(1/\alpha) -\boldsymbol{\mu} \delta \|u_1\|_p^{-2} \alpha^2}{(1 - C \alpha^n )^{2/p}}\|\psi_{\alpha}\|_p^2,
\end{aligned}
\end{equation*}
while if $n=4$, the result is
\begin{equation*}
\begin{aligned}
\int_X \psi_{\alpha} P_g \psi_{\alpha} \dd V_g &\leq \frac{b_n \mathcal{Y}(S^n, g_{\mathrm{std}}) + C \alpha^{2}+ C \alpha^2 -\boldsymbol{\mu} \delta \|u_1\|_p^{-2} \alpha^2\log(1/\alpha)}{(1 - C \alpha^n )^{2/p}}\|\psi_{\alpha}\|_p^2,
\end{aligned}
\end{equation*}
In any case,  as $\alpha\rightarrow 0$, the negative term involving $\boldsymbol{\mu}$ decays slower in absolute value than the other $\alpha$-dependent terms. Thus for $\alpha$ small enough, the negative term dominates so that we obtain
\begin{equation*}
  \int_X \psi_{\alpha} P_g \psi_{\alpha} \dd V_g < b_n \mathcal{Y}(S^n, g_{\mathrm{std}}) \|\psi_{\alpha}\|_p^2.
\end{equation*}
The theorem follows with a view on \eqref{PropMassDecomposition} and \eqref{MassFunctionalSphere}.
\end{proof}

\bibliography{MassTraceLit}

\end{document}